\documentclass[a4paper,twoside,11pt,reqno]{amsart}
\usepackage{fullpage}
\usepackage[T1]{fontenc}
\usepackage[utf8]{inputenc}

\usepackage{hyperref}
\usepackage[slantedGreeks, partialup, noDcommand]{kpfonts}
\usepackage{graphicx}
\usepackage[mathscr]{euscript}
\usepackage{caption}
\usepackage{tikz}
\usetikzlibrary{trees}

\hypersetup{ocgcolorlinks=true,allcolors=testc}
\hypersetup{
     colorlinks   = true,
     citecolor    = black
}
\hypersetup{linkcolor=black}
\hypersetup{urlcolor=black}

\newcommand{\eqdef}{\stackrel{\mathrm{def}}{=}}

\newtheorem{theorem}{Theorem}
\newtheorem{lemma}[theorem]{Lemma}

\newtheorem{proposition}[theorem]{Proposition}
\newtheorem{corollary}[theorem]{Corollary}

\newtheorem{remark}[theorem]{Remark}

\newtheorem*{definition*}{Definition}

\newcommand{\R}{\mathbb{R}} 
\newcommand{\Z}{\mathbb{Z}}
 
\newcommand{\N}{\mathbb{N}}

\newcommand{\e}{\varepsilon}

\usepackage{dutchcal}

\newcommand{\K}{\mathsf{K}}

\newcommand{\E}{\mathbb{E}}

\newcommand{\mb}{\mathbb}
\newcommand{\ms}{\mathscr}
\newcommand{\msf}{\mathsf}
\newcommand{\mr}{\mathrm}

\title{Approximation of polynomials from Walsh tail spaces}

\author{Alexandros Eskenazis}
\address{(A.~E.) CNRS, Institut de Math\'ematiques de Jussieu, Sorbonne Universit\'e, France and Trinity College, University of Cambridge, UK.}
\email{alexandros.eskenazis@imj-prg.fr, ae466@cam.ac.uk}

\author{Haonan Zhang}
\address{(H.~Z.)
Department of Mathematics,
University of South Carolina\\
Columbia, SC, 29208, USA.}
\email{haonanzhangmath@gmail.com, haonanz@mailbox.sc.edu}

\thanks{This material is based upon work supported by the NSF grant DMS-1929284 while the authors were in residence at ICERM for the Harmonic Analysis and Convexity program. }

\begin{document}

\maketitle

\begin{abstract}
We derive various bounds for the $L_p$ distance of polynomials on the hypercube from Walsh tail spaces, extending some of Oleszkiewicz's results (2017) for Rademacher sums.
\end{abstract}

\bigskip

{\footnotesize

}


\section{Introduction}

Given $n\in\N=\Z_{\ge 1}$, every function $f:\{-1,1\}^n\to\R$ admits a unique Fourier--Walsh expansion
\begin{equation}
\forall \ x\in\{-1,1\}^n,\qquad f(x) = \sum_{S\subseteq\{1,\ldots,n\}} \widehat{f}(S) w_S(x),
\end{equation}
where the Walsh function $w_S$ is given by $w_S(x) = \prod_{i\in S} x_i$ for $x=(x_1,\ldots,x_n) \in\{-1,1\}^n$. We shall say that $f$ is of degree at most $k\in\{1,\ldots,n\}$ if $\widehat{f}(S)=0$ for every subset $S$ of $\{1,\ldots,n\}$ with $|S|>k$. Similarly, we say that $f$ belongs on the $k$-th tail space, where $k\in\{1,\ldots,n\}$, if $\widehat{f}(S)=0$ for every subset $S$ with $|S|\leq k$. More generally, given a nonempty subset $I \subseteq\{0,1,\ldots,n\}$, we denote by
\begin{equation}
\ms{P}_I^n \eqdef \big\{ f:\{-1,1\}^n\to\R: \ \widehat{f}(S) = 0 \mbox{ for every } S \mbox{ with } |S|\notin I\big\}.
\end{equation}
We shall also adopt the natural notations $\ms{P}_{>k}^n = \ms{P}_{\{k+1,\ldots,n\}}^n$, $\ms{P}_{\leq k}^n = \ms{P}_{\{0,1,\ldots,k\}}^n$, $\ms{P}_{=k}^n=\ms{P}_{\{k\}}^n$ and so on.

Many modern developments in discrete analysis (see \cite{O'D14}) are centered around quantitative properties of functions whose spectrum is bounded above or below, in analogy with estimates established for polynomials in classical approximation theory on the torus $\mb{T}^n$ or on $\mb{R}^n$. One of the first results of this nature, going back at least to \cite{Bon70,Bou80}, is the important fact that all finite moments of low-degree Walsh polynomials are equivalent to each other up to dimension-free factors. Namely, given any $1\leq p\leq q <\infty$ and $k\in\N$, there exists a (sharp) constant $\msf{M}_{p,q}(k)$ such that for any $n\geq k$, every polynomial $f:\{-1,1\}^n\to \R$ of degree at most $k$ satisfies
\begin{equation} \label{eq:mom}
\|f\|_q \leq \msf{M}_{p,q}(k) \|f\|_p,
\end{equation}
where $\|\cdot\|_r$ always denotes the $L_r$ norm on $\{-1,1\}^n$ with respect to the uniform probability measure. Note that the reverse of \eqref{eq:mom} holds trivially with constant $1$ by H\"older's inequality. We refer to \cite{EI20,IT19,LS22} for the best known bounds on the implicit constant $\msf{M}_{p,q}(k)$. In the special case $k=1$, \eqref{eq:mom} is the celebrated Khintchine inequality \cite{Khi23} for Rademacher sums.

Our starting point is the simple observation that the moment comparison estimates \eqref{eq:mom} have the following (equivalent) dual formulation in terms of distances from tail spaces.

\begin{proposition} \label{prop:dual-mom}
For every $1\leq p\leq q< \infty$ and $k\in\N$, the constant $\msf{M}_{p,q}(k)$ in inequality \eqref{eq:mom} is also the least constant for which every function $f:\{-1,1\}^n\to \R$,  where $n\geq k$, satisfies
\begin{equation} \label{eq:dual-mom}
\inf_{g\in \ms{P}^n_{>k}} \|f-g\|_{p^\ast} \leq \msf{M}_{p,q}(k) \inf_{g\in \ms{P}^n_{>k}} \|f-g\|_{q^\ast},
\end{equation}
where the conjugate exponent $r^\ast$ of $r\in[1,\infty]$ satisfies $\frac{1}{r^\ast}+\frac{1}{r} = 1$.
\end{proposition}

Again, the reverse of \eqref{eq:dual-mom} holds with constant $1$. In the special case $k=0$, inequality \eqref{eq:dual-mom} becomes trivial with $\msf{M}_{p,q}(0)=1$ as both sides are equal to $|\mb{E}f|$. When $k=1$, which corresponds to the dual of the classical Khintchine inequality, we derive the following more precise formula for the distance from the tail space $\ms{P}_{>1}^n$.

\begin{theorem} \label{thm:dual-hit}
For every $1<r\leq \infty$ and $n\in\N$, every $f:\{-1,1\}^n\to\R$ satisfies\footnote{Throughout the paper we shall use standard asymptotic notation. For instance, $\xi\lesssim \eta$ (or $\eta\gtrsim \xi)$ means that there exists a universal constant $c>0$ such that $\xi \leq c\eta$ and $\xi\asymp\eta$ stands for $(\xi\lesssim \eta)\wedge (\eta\lesssim\xi)$. We shall use subscripts of the form $\lesssim_t, \gtrsim_t, \asymp_t$ when the implicit constant $c$ depends on some prespecified parameter $t$.}
\begin{equation}
\inf_{g\in \ms{P}^n_{>1}} \|f-g\|_{r} \asymp |\mb{E}f|+ \max_{i\in\{1,\ldots,n\}} \big| \widehat{f}(\{i\})\big|  + \sqrt{\frac{r-1}{r}} \Big(\sum_{i=1}^n \widehat{f}(\{i\})^2\Big)^{1/2}.
\end{equation}
\end{theorem}

This is the dual to a well-known result of Hitczenko \cite{Hit93} (see also \cite{Mon90,HK94}), obtaining $p$-independent upper and lower bounds for the $L_p$-norms of Rademacher sums, where $p\in[1,\infty)$.

At this point, we should point out that in both Proposition \ref{prop:dual-mom} and Theorem \ref{thm:dual-hit}, the exponents of the norms are always strictly greater than 1. For instance, choosing $f_1(x)=\sum_{i=1}^n x_i$, \eqref{eq:dual-mom} gives
\begin{equation}
\forall \ r\in(1,\infty],\qquad \inf_{g\in \ms{P}^n_{>k}} \|f_1-g\|_{r} \asymp_{r,k} \inf_{g\in \ms{P}^n_{>k}} \|f_1-g\|_{2} = \sqrt{n}. 
\end{equation}
On the other hand, it follows from a result of Oleszkiewicz \cite{Ole17}, which is the main precursor to this work, that the $L_1$-distance of $f_1$ from the $k$-th tail space satisfies
\begin{equation} \label{eq:ole0}
\inf_{g\in \ms{P}^n_{>k}} \|f_1-g\|_{1}  \asymp \min\{k,\sqrt{n}\},
\end{equation}
and thus exhibits a starkly different behavior as $n\to\infty$ from the $L_r$ norms with $r>1$.

More generally, it is shown in \cite{Ole17} that for every $a_1\geq \cdots \geq a_n\geq0$, we have
\begin{equation} \label{eq:ole}
\inf_{g\in \ms{P}^n_{>k}} \| f_{\pmb{a}} -g\|_{1}  \asymp \min_{r\in\{0,1,\ldots,n\}}\bigg\{ \Big(\sum_{i=1}^r a_i^2\Big)^{1/2} +k a_{r+1}\bigg\},
\end{equation}
where for $\pmb{a}=(a_1,\ldots,a_n)$ we denote $f_{\pmb{a}}(x) = \sum_{i=1}^n a_i x_i$ and we make the convention that $a_{n+1}=0$.  The quantity appearing on the right hand side of \eqref{eq:ole} can be rephrased in terms of the $\K$-functional of real interpolation (see \cite[Chapter~3]{BL76}). Recall that if $(A_0,A_1)$ is an interpolation pair, then the Lions--Peetre $\K$-functional is defined for every $t\geq0$ and $a\in A_0+A_1$ as
\begin{equation} \label{def:K}
\K(a,t;A_0,A_1) \eqdef \inf\big\{ \|a_0\|_{A_0}+t\|a_1\|_{A_1}: \ a=a_0+a_1\big\}.
\end{equation}
It is elementary to check (see \cite{Hol70}),  that if $a_1\geq\cdots\geq a_n\geq0$ and $k\in\N$,  then
\begin{equation} \label{eq:ele}
\min_{r\in\{0,1,\ldots,n\}}\bigg\{ \Big(\sum_{i=1}^r a_i^2\Big)^{1/2} +k a_{r+1}\bigg\} \asymp \K(\pmb{a},k;\ell_2^n,\ell_\infty^n).
\end{equation}
Note that the right-hand side is invariant under permutations of the entries of $\pmb{a}$.
The main result of this work is an appropriate extension of the upper bound in Oleszkiewicz's result \eqref{eq:ole} to polynomials of arbitrary degree on the discrete hypercube.

\begin{theorem} \label{thm:main}
For every $d\in\N$, there exists $\msf{C}_d\in(0,\infty)$ such that for any $n\geq k\geq d$,  every polynomial $f:\{-1,1\}^n\to \R$ of degree at most $d$ satisfies 
\begin{equation} \label{eq:main}
\inf_{g\in \ms{P}_{>k}^n} \|f-g\|_1 \leq \K \big( \widehat{f}, \msf{C}_d k^d; \ell_2^m, \ell_{\frac{2d}{d-1}}^m\big),
\end{equation}
where $\widehat{f}$ is the vector of Fourier coefficients of $f$, viewed as an element of $\R^m$ with $m=\binom{n}{0}+\cdots+\binom{n}{d}$.
\end{theorem}

As was already pointed out by Oleszkiewicz,  the method of \cite{Ole17} does not appear to extend beyond Rademacher sums.  Instead,  in our proof we shall employ the discrete \emph{Bohnenblust--Hille} inequality from approximation theory (see \cite{Ble01,DS14,DMP19,DGMS19}) along with a classical bound of Figiel on the Rademacher projection of polynomials.  A discussion concerning the size of the implicit constant $\msf{C}_d$ appearing in \eqref{eq:main} is postponed to Section \ref{sec:2} (see Remark \ref{rem:d} there).

Unlike the two-sided inequality \eqref{eq:ole}, our bound \eqref{eq:main} is only one-sided and as a matter of fact there are examples in which it is far from optimal.  In particular, for functions which are permutationally symmetric, we obtain a more accurate estimate.  In what follows, we shall denote by $T_k(x) = \sum_{\ell=0}^k c(k,\ell) x^\ell$ the $k$-th Chebyshev polynomial of the first kind characterized by the property $T_k(\cos \theta) = \cos(k\theta)$, where $\theta\in\R$.  Moreover, we shall use the ad hoc notation
\begin{equation} \label{eq:ctil}
\tilde{c}(k,\ell) \eqdef \begin{cases} c(k,\ell), & \mbox{ if } k-\ell \mbox{ is even} \\ c(k-1, \ell),& \mbox{ if } k-\ell \mbox{ is odd} \end{cases}. 
\end{equation}
For $\ell\in\{1,\ldots,n\}$, let $f_\ell$ be the $\ell$-th elementary symmetric multilinear polynomial
\begin{equation}
\forall \ x\in\{-1,1\}^n, \qquad f_\ell(x) \eqdef\sum_{\substack{S\subseteq\{1,\ldots,n\}: \\  |S|=\ell}} w_S(x).
\end{equation}
We have the following bound on the distance of symmetric polynomials from tail spaces.

\begin{theorem} \label{thm:symmetric}
Let $n,k,d\in\N$ with $n\geq k\geq d$. Then, every symmetric polynomial 
\begin{equation}
f = \sum_{\ell=0}^d \alpha_\ell f_\ell
\end{equation}
of degree at most $d$ on $\{-1,1\}^n$ satisfies
\begin{equation} \label{eq:sym}
\inf_{g\in \ms{P}_{>k}^n} \|f-g\|_1 \leq \sum_{\ell=0}^d |\alpha _\ell| | \tilde{c}(k,\ell)|.
\end{equation}
\end{theorem}

This bound can sometimes be reversed and, in particular, it gives a sharp estimate as $n\to\infty$ for the $L_1$-distance of the elementary symmetric polynomial $f_d$ from the $k$-th tail space.

\begin{corollary} \label{cor}
For every $n,k,d\in\N$ with $n\geq k \geq d$, there exists $\e_n(k,d)>0$ such that
\begin{equation} \label{eq:cor}
| \tilde{c}(k,d)| - \e_n(k,d)  \leq \inf_{g\in \ms{P}_{>k}^n} \|f_d-g\|_1 \leq | \tilde{c}(k,d)|
\end{equation}
and $\lim_{n\to\infty} \e_n(k,d)=0$.
\end{corollary}

The main motivation behind the work \cite{Ole17} was a question of Bogucki, Nayar and Wojciechowski, asking to estimate the $L_1$-distance of the Rademacher sum $f_1$ from the $k$-th tail space.  Corollary \ref{cor} extends (at least asymptotically in $n$) the answer given by Oleszkiewicz to all symmetric homogeneous polynomials.  We point out though that for $k=1$, \eqref{eq:cor} is sharper than Oleszkiewicz's bound \eqref{eq:ole0} as $n\to\infty$, as \eqref{eq:ole0} is tight only up to a multiplicative constant.


\subsection*{Acknowledgements} We are grateful to Krzysztof Oleszkiewicz for valuable discussions. H. Z. is grateful to Institut de Mathématiques de Jussieu for the hospitality during a visit in 2023.


\section{Proofs} \label{sec:2}

We proceed to the proofs of our results. We start with the simple duality argument leading to Proposition \ref{prop:dual-mom},  variants of which will be used throughout the paper.

\begin{proof}[Proof of Proposition \ref{prop:dual-mom}]
Consider the identity operator acting as $\msf{id}(h)=h$ on a function of the form $h:\{-1,1\}^n\to\R$.  Then, the optimal constant $\msf{M}_{p,q}(k)$ can be expressed as
\begin{equation} \label{m*}
\msf{M}_{p,q}(k) = \big\| \msf{id}: (\ms{P}_{\leq k}^n,\|\cdot\|_p) \to (\ms{P}_{\leq k}^n,\|\cdot\|_q) \big\| = \big\| \msf{id}^*: (\ms{P}_{\leq k}^n,\|\cdot\|_q)^\ast \to (\ms{P}_{\leq k}^n,\|\cdot\|_p)^\ast \big\|
\end{equation}
by duality. Moreover, observe that since $(\ms{P}_{\leq k}^n,\|\cdot\|_r)$ is a subspace of $L_r$, its dual is isometric to
\begin{equation}
(\ms{P}_{\leq k}^n,\|\cdot\|_r)^\ast = L_{r^\ast} \big/(\ms{P}_{\leq k}^n)^\perp = L_{r^\ast} \big/\ms{P}_{> k}^n,
\end{equation}
where $A^\perp$ is the annihilator of $A$.  Since it is also clear that $\msf{id}^* = \msf{id}$,  \eqref{m*} concludes the proof.
\end{proof}

Using a theorem of Hitczenko \cite{Hit93} as input and the same duality, we deduce Theorem \ref{thm:dual-hit}.

\begin{proof}[Proof of Theorem \ref{thm:dual-hit}]
The result of \cite{Hit93} asserts that if $\pmb{a}=(a_0,a_1,\ldots,a_n)$ and $f_{\pmb{a}}(x) = a_0+\sum_{i=1}^n a_ix_i$, 
\begin{equation}
\|f_{\pmb{a}}\|_{r^\ast} = \bigg( \mb{E} \Big|\sum_{i=0}^n a_i x_i \Big|^{r^*}\bigg)^{1/r^*} \asymp \K(\pmb{a},\sqrt{r^\ast}; \ell_1^{n+1}, \ell_2^{n+1}),
\end{equation}
where $x_0, x_1,\ldots,x_n$ are independent Bernoulli random variables, and the first equality holds due to symmetry.  In other words, the linear operator 
\begin{equation}
T: \big(\R^{n+1},  \K(\cdot,\sqrt{r^\ast}; \ell_1^{n+1}, \ell_2^{n+1})\big) \to (\ms{P}_{\leq 1}^n, \|\cdot\|_{r^\ast})
\end{equation}
given by $T{\pmb{a}}=f_{\pmb{a}}$  is an isomorphism, and thus the same holds for its adjoint.  Recalling that
\begin{equation}
\K(\pmb{a},\sqrt{r^\ast}; \ell_1^{n+1}, \ell_2^{n+1}) = \inf\big\{ \|\pmb{b}\|_{\ell_1^{n+1}} + \sqrt{r^\ast} \|\pmb{c}\|_{\ell_2^{n+1}}: \ \pmb{a}=\pmb{b}+\pmb{c}\big\}
\end{equation}
and the duality between sums and intersections of normed spaces \cite[Theorem~2.7.1]{BL76}, we see that the dual space of $(\R^{n+1},  \K(\cdot,\sqrt{r^\ast}; \ell_1^{n+1}, \ell_2^{n+1}))$ can be identified with
\begin{equation}
\forall \ y\in\R^{n+1}, \qquad \|y\|_{(\R^{n+1},  \K(\cdot,\sqrt{r^\ast}; \ell_1^{n+1}, \ell_2^{n+1}))^\ast} = \max\Big\{ \|y\|_{\ell_\infty^{n+1}},  \frac{\|y\|_{\ell_2^{n+1}}}{\sqrt{r^\ast}}\Big\}.
\end{equation}
By Parseval's identity,  the action of the adjoint 
\begin{equation}
T^\ast: L_r \big/ \ms{P}_{>1}^n \to (\R^{n+1},  \K(\cdot,\sqrt{r^\ast}; \ell_1^{n+1}, \ell_2^{n+1}))^\ast
\end{equation}
is given by
\begin{equation}
T^\ast( f + \ms{P}_{>1}^n ) = \big(\mb{E}f,  \widehat{f}(\{1\}) , \ldots, \widehat{f}(\{n\}) \big)
\end{equation}
and thus the conclusion is equivalent to fact that $T^\ast$ is an isomorphism. 
\end{proof}

We now proceed to the proof of the general upper bound for polynomials given in Theorem \ref{thm:main}. The first ingredient for the proof is a discrete version of the classical Bohnenblust--Hille inequality from approximation theory (see the survey \cite{DS14}) proven in \cite{Ble01,DMP19}. This asserts that for every $d\in \N$, there exists a (sharp) constant $\msf{B}_d\in(0,\infty)$ such that for any $n\geq d$, every polynomial $f:\{-1,1\}^n\to \R$ of degree at most $d$ satisfies
\begin{equation} \label{eq:bh}
\Big( \sum_{S\subseteq\{1,\ldots,n\}} | \widehat{f}(S)|^{\frac{2d}{d+1}} \Big)^{\frac{d+1}{2d}} \leq \msf{B}_d \|f\|_\infty.
\end{equation}
Moreover, $\frac{2d}{d+1}$ is the least exponent for which the implicit constant  becomes independent of the ambient dimension $n$. The best known upper bound $B_d \leq \exp(C\sqrt{d\log d})$ for the constant $\msf{B}_d$ is due to the work of Defant, Masty\l o and P\'erez \cite{DMP19}. 



The level $\ell$-Rademacher projection of a function $f:\{-1,1\}^n\to\R$ is defined as
\begin{equation}
\forall \ x\in\{-1,1\}^n,\qquad \mr{Rad}_\ell f(x) \eqdef \sum_{\substack{S\subseteq\{1,\ldots,n\}: \\ |S|= \ell}} \widehat{f}(S) w_S(x).
\end{equation}
Moreover, we write $\mr{Rad}_{\leq d} = \sum_{\ell\leq d} \mr{Rad}_\ell$.  Apart from the discrete Bohnenblust--Hille inequality \eqref{eq:bh}, we will also use a standard bound on the norm of the $\ell$-Rademacher projections which is usually attributed to Figiel (see also \cite[Section~3]{EI22} for a short proof).

\begin{proposition} \label{prop:figiel}
Let $n\geq k\geq d$. Then, every function $f:\{-1,1\}^n\to\R$ of degree at most $k$ satisfies
\begin{equation} \label{eq:figiel0}
\forall\ 0\le \ell \le d, \qquad \big\|\mr{Rad}_\ell f\big\|_\infty \leq 
|\tilde{c}(k,\ell)| \ \|f\|_\infty,
\end{equation}
and thus,
\begin{equation} \label{eq:figiel}
\big\| \mr{Rad}_{\leq d} f\big\|_\infty \leq \sum_{\ell=0}^d \big\|\mr{Rad}_\ell f\big\|_\infty \leq 
\sum_{\ell=0}^d |\tilde{c}(k,\ell)| \ \|f\|_\infty,
\end{equation}
where $\tilde{c}(k,\ell)$ is given by \eqref{eq:ctil}.  It is moreover known that $|\tilde{c}(k,\ell)|\le \frac{k^\ell}{\ell!}$.
\end{proposition}

Combining the above with Parseval's identity, we deduce the following bound.

\begin{lemma} \label{lem:dual-main}
Let $n\geq k \geq d$. Then, every function $f:\{-1,1\}^n\to\R$ of degree at most $k$ satisfies
\begin{equation}
\max\big\{ \big\| \widehat{\mr{Rad}_{\leq d}(f)} \big\|_{\ell_2^{m}},  \sigma(k,d)^{-1} \big\| \widehat{\mr{Rad}_{\leq d}(f)} \big\|_{\ell_{\frac{2d}{d+1}}^{m}} \big\} \leq \inf_{g\in\ms{P}_{>d}^n\cap\ms{P}_{\leq k}^n} \|f-g\|_\infty,
\end{equation}
where $m=\binom{n}{0}+\cdots+\binom{n}{d}$ and $\sigma(k,d) = \msf{B}_d\sum_{\ell=0}^d |\tilde{c}(k,\ell)|$.
\end{lemma}

\begin{proof}
Fix a function $g\in\ms{P}_{>d}^n\cap\ms{P}_{\leq k}^n$.  Then, 
\begin{equation}
 \big\| \widehat{\mr{Rad}_{\leq d}(f)} \big\|_{\ell_2^{m}} \leq  \big\| \widehat{f} - \widehat{g} \big\|_{\ell_2^{M}} = \|f-g\|_2 \leq \|f-g\|_\infty,
\end{equation}
where $M=\binom{n}{0}+\cdots+\binom{n}{k}$.  Moreover,  we have
\begin{equation*}
 \big\| \widehat{\mr{Rad}_{\leq d}(f)} \big\|_{\ell_{\frac{2d}{d+1}}^{m}} \stackrel{\eqref{eq:bh}}{\leq} \msf{B}_d \big\| \mr{Rad}_{\leq d}(f) \big\|_{\infty} =  \msf{B}_d \big\| \mr{Rad}_{\leq d}(f-g) \big\|_{\infty}  \stackrel{\eqref{eq:figiel}}{\leq} \msf{B}_d\sum_{\ell=0}^d |\tilde{c}(k,\ell)| \ \|f-g\|_\infty. \hfill\qedhere
\end{equation*}
\end{proof}

Equipped with Lemma \ref{lem:dual-main}, we can complete the proof of Theorem \ref{thm:main}.

\begin{proof}[Proof of Theorem \ref{thm:main}] Consider the normed spaces $X = (\ms{P}_{\leq k}^n , \|\cdot\|_\infty)$ and $Y=(\R^m,\|\cdot\|_Y)$ with
\begin{equation}
\forall \ y\in \R^m, \qquad \|y\|_Y = \max\big\{\|y\|_{\ell_2^m},   \sigma(k,d)^{-1} \|y\|_{\ell_{\frac{2d}{d+1}}^m}\big\}
\end{equation}
and $m=\binom{n}{0}+\cdots+\binom{n}{d}$. Moreover,  let $Z = \ms{P}^n_{>d}\cap\ms{P}_{\leq k}^n \subset X$,  viewed as a normed subspace of $X$. Lemma \ref{lem:dual-main} asserts that the linear operator $A: X/Z \to Y$ given by
\begin{equation}
\forall \ f\in X, \qquad A( f + Z ) = \big( \widehat{f}(S) \big)_{\substack{|S|\leq d}}
\end{equation}
has norm $\|A\| \leq 1$. Therefore, the same holds for its adjoint $A^*: Y^* \to (X/Z)^*$.

By the usual duality between sums and intersections of normed spaces \cite[Theorem~2.7.1]{BL76},  we see that the space $Y^*$ is isometric to
\begin{equation}
\forall \ w\in\R^m, \qquad \|w\|_{Y^\ast} = \msf{K}\big(w, \sigma(k,d); \ell_2^m, \ell_{\frac{2d}{d-1}}^m\big).
\end{equation}
Moreover, as $X/Z$ is a quotient of $X$, its dual is the subspace of $X^* = L_1 / \ms{P}^n_{>k}$ which is identified with the annihilator of $Z$ inside $X^\ast$.  In other words,  it is the set
\begin{equation}
(X/Z)^\ast = \big\{ f+\ms{P}_{>k}^n: \ \mb{E}[fg]=0 \ \mbox{for every } g\in Z \big\} =  \big\{ f+\ms{P}_{>k}^n: \ f\in\ms{P}_{\leq d}^n\big\}= \mr{span}(\ms{P}_{\leq d}^n \cup \ms{P}_{>k}^n) \big/ \ms{P}_{>k}^n
\end{equation}
equipped with the $L_1$ quotient norm.  Finally, for a sequence $\pmb{a}=(a_S)_{|S|\leq d}\in Y^\ast$ and an equivalence class $f+Z\in X/Z$,  we have
\begin{equation}
\langle \pmb{a}, A(f+Z) \rangle = \sum_{\substack{S\subseteq\{1,\ldots,n\}: \\ |S|\leq d}} a_S \widehat{f}(S) = \Big\langle \sum_{\substack{S\subseteq\{1,\ldots,n\}: \\ |S|\leq d}} a_S w_S + \ms{P}_{>k}^n ,  f+Z \Big\rangle =  \langle A^*(\pmb{a}) , f+Z \rangle,
\end{equation}
where the first brackets $\langle \cdot, \cdot\rangle$ denote the duality in $Y$ and the following brackets denote the duality in $X/Z$.  Therefore, we conclude that 
\begin{equation}
\forall  \ \pmb{a}\in Y^*, \qquad A^*(\pmb{a}) = \sum_{\substack{S\subseteq\{1,\ldots,n\}: \\ |S|\leq d}} a_S w_S + \ms{P}_{>k}^n
\end{equation}
and thus, the condition $\|A^*\|\leq 1$ means that for any $f:\{-1,1\}^n\to\R$ of degree at most $d$,
\begin{equation}
\inf_{g\in\ms{P}_{>k}^n} \|f-g\|_1 = \big\| A^*\big(\widehat{f}\big) \big\|_{(X/Z)^*} \leq \big\|\widehat{f}\big\|_{Y^*} = \msf{K}\big( \widehat{f}, \sigma(k,d); \ell_2^m, \ell_{\frac{2d}{d-1}}^m\big).
\end{equation}
Finally, since
\begin{equation}
\sigma(k,d) \leq \msf{B}_d \sum_{\ell=0}^d |\tilde{c}(k,\ell)| \leq \msf{B}_d \sum_{\ell=0}^d \frac{k^\ell}{\ell!} \leq e\msf{B}_d k^d,
\end{equation}
we deduce the conclusion of the theorem with $\msf{C}_d=e\msf{B}_d$.
\end{proof}

\begin{remark} \label{rem:d}
To the best of our knowledge, there are no nonconstant lower bounds on the size of the discrete Bohnenblust--Hille constant $\msf{B}_d$, so it is even conceivable that the constant $\msf{C}_d$ in \eqref{eq:main} can be chosen to be independent of $d$.
\end{remark}

\begin{remark}
A duality argument similar to that employed for Theorem \ref{thm:main} shows that for every $d\in\N$, the constant $\msf{B}_d$ in inequality \eqref{eq:bh} is also the least constant for which every function $f:\{-1,1\}^n\to\R$, where $n\geq d$, satisfies
\begin{equation} \label{eq:dual-bh}
\inf_{g\in \ms{P}_{>d}^n} \|f-g\|_{1} \leq \msf{B}_d \bigg( \sum_{\substack{S\subseteq\{1,\ldots,n\}: \\ |S|\leq d}} |\widehat{f}(S)|^{\frac{2d}{d-1}} \bigg)^{\frac{d-1}{2d}}.
\end{equation}
\end{remark}

\begin{remark}
It was pointed out to us by Oleszkiewicz that the main result \eqref{eq:ole} of \cite{Ole17} also admits a dual formulation. Namely, for every $\pmb{a}=(a_1,\ldots,a_n)\in\R^n$,  we have
\begin{equation}
\inf \big\{ \|f_{\pmb{a}} - g\|_\infty: \ g\in \ms{P}_{\{0\}\cup\{2,\ldots,k\}}^n\big\} \asymp \max\Big\{ \|\pmb{a}\|_{\ell_2^n}, \frac{\|\pmb{a}\|_{\ell_1^n}}{k}\Big\}.
\end{equation}
This can be proven using similar ideas as in the proof of Theorem \ref{thm:main}.
\end{remark}

A slight variant of the arguments above also yields Theorem \ref{thm:symmetric} for symmetric functions.

\begin{proof}[Proof of Theorem \ref{thm:symmetric}]
Let $f$ be a symmetric function of the form $f=\sum_{\ell=0}^d \alpha_\ell f_\ell$ where $f_\ell$ is the $\ell$-th elementary symmetric polynomial. Then,  the Hahn--Banach theorem gives
\begin{equation} \label{eq:hb}
\inf_{g\in \ms{P}_{>k}^n} \|f-g\|_1 = \sup_{0\neq h\in \ms{P}_{\leq k}^n} \frac{\mb{E} [fh]}{\|h\|_\infty}.
\end{equation}
Observe now that we can write
\begin{equation}
\mb{E}[fh] = \sum_{\ell=0}^d \alpha_\ell \mb{E}[f_\ell h] = \sum_{\ell=0}^d \alpha_\ell \sum_{\substack{S\subseteq\{1,\ldots,n\}: \\  |S|=\ell}} \widehat{h}(S) = \sum_{\ell=0}^d \alpha_\ell \mr{Rad}_\ell h(1,\ldots,1).
\end{equation}
Thus, by Figiel's bound \eqref{eq:figiel0},
\begin{equation}
\mb{E}[fh] \leq \sum_{\ell=0}^d |\alpha_\ell| \ \|\mr{Rad}_\ell h\|_\infty \leq \sum_{\ell=0}^d |\alpha_\ell| |\tilde{c}(k,\ell)|\ \|h\|_\infty
\end{equation}
and the desired inequality follows from \eqref{eq:hb}.
\end{proof}

Equipped with Theorem \ref{thm:symmetric}, we present the proof of Corollary \ref{cor}.

\begin{proof}[Proof of Corollary \ref{cor}]
The upper bound in \eqref{eq:cor} follows immediately from Theorem \ref{thm:symmetric}.  For the lower bound, consider the auxiliary symmetric function $H_{k,n}:\{-1,1\}^n\to \R$ given by
\begin{equation}
\forall \ x\in\{-1,1\}^n,\qquad H_{k,n}(x)\eqdef T_k\left(\frac{x_1+\cdots +x_n}{n}\right)=\sum_{\ell=0}^{k}\beta_{\ell,k,n} \ f_\ell(x),
\end{equation}
where $f_\ell$ is the $\ell$-th elementary symmetric polynomial, and notice that $H_{k,n}$ has degree at most $k$.  As $T_k(x)$ is odd or even when $k$ is odd or even respectively, it follows that $\beta_{\ell,k,n}=0$ if $k-\ell$ is odd.  We distinguish two cases depending on the parity of $k-d$.

\smallskip

\noindent $\bullet$ Suppose that $k-d$ is even and consider the function $\varphi_{d,k,n}:\{-1,1\}^n\to \R$ given by
\begin{equation}\label{eq:f_d,k,n}
\forall \ x\in\{-1,1\}^n,\qquad	\varphi_{d,k,n}(x)\eqdef \sum_{0\le \ell \le d:\ 2|d-\ell}\mr{sign}(\beta_{\ell,k,n}) f_\ell(x)
\end{equation}
that is also symmetric and of degree at most $d$. Then,  on one hand we know that 
\begin{equation} \label{ubb}
	\inf_{g\in \ms{P}_{>k}^n}\|\varphi_{d,k,n}-g\|_1
	\stackrel{\eqref{eq:sym}}{\le} \sum_{0\le \ell\le d:\ 2|d-\ell}|\tilde{c}(k,\ell)| 
	=\sum_{0\le \ell\le d:\ 2|d-\ell}|c(k,\ell)| .
\end{equation}
On the other hand, we have the following lower estimate,
\begin{equation}
		\inf_{g\in \ms{P}_{>k}^n} \|\varphi_{d,k,n}-g\|_1 \stackrel{\eqref{eq:hb}}{=} \sup_{0\neq h\in \ms{P}_{\leq k}^n} \frac{\mb{E} [\varphi_{d,k,n}h]}{\|h\|_\infty}
		\ge \frac{|\E[\varphi_{d,k,n}H_{k,n}]|}{\|H_{k,n}\|_\infty}.
\end{equation}
By definition, $\|H_{k,n}\|_\infty\le \sup_{x\in [-1,1]}|T_k(x)|\le 1$ and $H_{k,n}(1,\ldots,1)=T_k(1)=1$.  Therefore, 
\begin{equation}
\begin{split}
	\inf_{g\in \ms{P}_{>k}^n} \|\varphi_{d,k,n}-g\|_1
	 \ge |\E[\varphi_{d,k,n}H_{k,n}]|
	 &=\bigg|\sum_{0\le \ell\le d:\ 2|d-\ell}\mr{sign}(\beta_{\ell,k,n}) \mr{Rad}_\ell H_{k,n}(1,\dots, 1)\bigg|\\
	 &=\sum_{0\le \ell\le d:\ 2|d-\ell}\left| \mr{Rad}_\ell H_{k,n}(1,\dots, 1)\right|.
\end{split}
\end{equation}
To further estimate this sum, we use \cite[Lemma~27]{IRRRY21} which implies that there exists a positive constant $\e_n(k,d)>0$ with $\e_{n}(k,d) = O_{k,d}(1/n)$ as $n\to\infty$, such that
\begin{equation}
\sum_{0\le \ell\le d:\ 2|d-\ell}\left |\mr{Rad}_\ell H_{k,n}(1,\dots, 1)\right|
	\ge\sum_{0\le \ell\le d:\ 2|d-\ell} |c(k,\ell)|-\e_n(k,d).
\end{equation}
Hence,  combining the above we conclude that
\begin{equation}\label{ineq: sharp2}
			\inf_{g\in \ms{P}_{>k}^n}\|\varphi_{d,k,n}-g\|_1
			\ge\sum_{0\le \ell\le d:\ 2|d-\ell} |c(k,\ell)|-\e_n(k,d).
\end{equation}
Finally, to bound from below the $L_1$-distance of $f_d$ from the tail space,  we write (putting $\varphi_{0,k,n}=\varphi_{-1,k,n}\equiv 0$)
\begin{equation}
	f_d
	=\mr{sign}(\beta_{d,k,n})(\varphi_{d,k,n}-\varphi_{d-2,k,n})
\end{equation}
and using the triangle inequality, we get
\begin{equation}
\begin{split}
	 		\inf_{g\in \ms{P}_{>k}^n}&\left\|f_d-g\right\|_{1}
	 		\ge \inf_{g\in \ms{P}_{>k}^n}\left\|\varphi_{d,k,n}-g\right\|_{1}
	 		-\inf_{g\in \ms{P}_{>k}^n}\left\|\varphi_{d-2,k,n}-g\right\|_{1}\\
	 		&\stackrel{\eqref{ubb}\wedge\eqref{ineq: sharp2}}{\ge} \sum_{0\le \ell\le d:\ 2|d-\ell} |c(k,\ell)| - \e_n(k,d) -\sum_{0\le \ell\le d-2:\ 2|d-2-\ell} |c(k,\ell)| =|c(k,d)|-\e_n(k,d),
\end{split}
\end{equation}
thus concluding the proof of the lower bound in \eqref{eq:cor}.

\smallskip

\noindent $\bullet$ If $k-d$ is odd, we use the identity (putting $\varphi_{0,k,n}=\varphi_{-1,k,n}=\varphi_{d,-1,n}\equiv 0$)
\begin{equation}
f_d = \mr{sign}(\beta_{d,k-1,n})(\varphi_{d,k-1,n}-\varphi_{d-2,k-1,n}).
\end{equation}
The rest of the argument is identical.
\end{proof}

\begin{remark}
In this paper, we studied dual versions of moment comparison estimates on the hypercube \eqref{eq:mom} and investigated the endpoint case of their duals \eqref{eq:dual-mom} for polynomials.  By formal reasoning similar to the proof of Proposition \ref{prop:dual-mom}, one can derive dual versions of various other polynomial inequalities,  including Bernstein--Markov inequalities and their reverses and bounds for the action of the heat semigroup. We refer to \cite{EI20} for  a systematic treatment of such estimates.
\end{remark}




\begin{thebibliography}{10}

\bibitem{BL76}
J\"{o}ran Bergh and J\"{o}rgen L\"{o}fstr\"{o}m.
\newblock {\em Interpolation spaces. {A}n introduction}, volume No. 223 of {\em
  Grundlehren der Mathematischen Wissenschaften}.
\newblock Springer-Verlag, Berlin-New York, 1976.

\bibitem{Ble01}
Ron Blei.
\newblock {\em Analysis in integer and fractional dimensions}, volume~71 of
  {\em Cambridge Studies in Advanced Mathematics}.
\newblock Cambridge University Press, Cambridge, 2001.

\bibitem{Bon70}
Aline Bonami.
\newblock \'{E}tude des coefficients de {F}ourier des fonctions de
  {$L^{p}(G)$}.
\newblock {\em Ann. Inst. Fourier (Grenoble)}, 20(fasc. 2):335--402 (1971),
  1970.

\bibitem{Bou80}
Jean Bourgain.
\newblock Walsh subspaces of {$L^{p}$}-product spaces.
\newblock In {\em Seminar on {F}unctional {A}nalysis, 1979--1980 ({F}rench)},
  pages Exp. No. 4A, 9. \'{E}cole Polytech., Palaiseau, 1980.

\bibitem{DGMS19}
Andreas Defant, Domingo Garc\'{\i}a, Manuel Maestre, and Pablo Sevilla-Peris.
\newblock {\em Dirichlet series and holomorphic functions in high dimensions},
  volume~37 of {\em New Mathematical Monographs}.
\newblock Cambridge University Press, Cambridge, 2019.

\bibitem{DMP19}
Andreas Defant, Mieczys{\l}aw Masty{\l}o, and Antonio P\'{e}rez.
\newblock On the {F}ourier spectrum of functions on {B}oolean cubes.
\newblock {\em Math. Ann.}, 374(1-2):653--680, 2019.

\bibitem{DS14}
Andreas Defant and Pablo Sevilla-Peris.
\newblock The {B}ohnenblust-{H}ille cycle of ideas from a modern point of view.
\newblock {\em Funct. Approx. Comment. Math.}, 50(1, [2013 on table of
  contents]):55--127, 2014.

\bibitem{EI20}
Alexandros Eskenazis and Paata Ivanisvili.
\newblock Polynomial inequalities on the {H}amming cube.
\newblock {\em Probab. Theory Related Fields}, 178(1-2):235--287, 2020.

\bibitem{EI22}
Alexandros Eskenazis and Paata Ivanisvili.
\newblock Learning low-degree functions from a logarithmic number of random
  queries.
\newblock In {\em S{TOC} '22---{P}roceedings of the 54th {A}nnual {ACM}
  {SIGACT} {S}ymposium on {T}heory of {C}omputing}, pages 203--207. ACM, New
  York, [2022] \copyright 2022.

\bibitem{Hit93}
Pawe{\l} Hitczenko.
\newblock Domination inequality for martingale transforms of a {R}ademacher
  sequence.
\newblock {\em Israel J. Math.}, 84(1-2):161--178, 1993.

\bibitem{HK94}
Pawe{\l} Hitczenko and Stanis{\l}aw Kwapie\'{n}.
\newblock On the {R}ademacher series.
\newblock In {\em Probability in {B}anach spaces, 9 ({S}andjberg, 1993)},
  volume~35 of {\em Progr. Probab.}, pages 31--36. Birkh\"{a}user Boston,
  Boston, MA, 1994.

\bibitem{Hol70}
Tord Holmstedt.
\newblock Interpolation of quasi-normed spaces.
\newblock {\em Math. Scand.}, 26:177--199, 1970.

\bibitem{IT19}
Paata Ivanisvili and Tomasz Tkocz.
\newblock Comparison of moments of {R}ademacher chaoses.
\newblock {\em Ark. Mat.}, 57(1):121--128, 2019.

\bibitem{IRRRY21}
Siddharth Iyer, Anup Rao, Victor Reis, Thomas Rothvoss, and Amir Yehudayoff.
\newblock Tight bounds on the {F}ourier growth of bounded functions on the
  hypercube.
\newblock ECCC 2021, Preprint available at
  \url{https://arxiv.org/abs/2107.06309}, 2021.

\bibitem{Khi23}
Aleksandr Khintchine.
\newblock \"{U}ber dyadische {B}r\"{u}che.
\newblock {\em Math. Z.}, 18(1):109--116, 1923.

\bibitem{LS22}
Niv Levhari and Alex Samorodnitsky.
\newblock Hypercontractive inequalities for the second norm of highly
  concentrated functions, and {M}rs. {G}erber's-type inequalities for the
  second {R}\'{e}nyi entropy.
\newblock {\em Entropy}, 24(10):Paper No. 1376, 27, 2022.

\bibitem{Mon90}
Stephen~J. Montgomery-Smith.
\newblock The distribution of {R}ademacher sums.
\newblock {\em Proc. Amer. Math. Soc.}, 109(2):517--522, 1990.

\bibitem{O'D14}
Ryan O'Donnell.
\newblock {\em Analysis of {B}oolean functions}.
\newblock Cambridge University Press, New York, 2014.

\bibitem{Ole17}
Krzysztof Oleszkiewicz.
\newblock On mimicking {R}ademacher sums in tail spaces.
\newblock In {\em Geometric aspects of functional analysis}, volume 2169 of
  {\em Lecture Notes in Math.}, pages 331--337. Springer, Cham, 2017.

\end{thebibliography}
\end{document}